\documentclass[12pt,reqno]{amsart}
\title{Extensions of invariant random orders on groups}
\usepackage{amsmath, amsthm, amsopn, amssymb, microtype}
\usepackage{mathrsfs} 
\usepackage{mathscinet}
\usepackage[normalem]{ulem}
\usepackage{enumerate, tikz, etoolbox, intcalc, geometry, caption, subcaption, afterpage}
\usepackage[english]{babel}
\usepackage{mathtools}
\usepackage{enumitem}

\author{Yair Glasner} 
\address{Ben-Gurion University of the Negev.
	Departement of Mathematics.
	Be'er Sheva, 8410501, Israel}

\author{Frank Lin}
\address{Ben-Gurion University of the Negev.
	Departement of Mathematics.
	Be'er Sheva, 8410501, Israel}

\author{Tom Meyerovitch}
\address{Ben-Gurion University of the Negev.
	Departement of Mathematics.
	Be'er Sheva, 8410501, Israel}

\usepackage[colorlinks=true,urlcolor=blue,pdfborder={0 0 0}]{hyperref}
\hypersetup{linkcolor=[rgb]{0,0,0.6}}
\hypersetup{citecolor=[rgb]{0,0.6,0}}
\usepackage[nameinlink]{cleveref}
\crefname{theorem}{Theorem}{Theorems}
\crefname{thm}{Theorem}{Theorems}
\crefname{mainthm}{Theorem}{Theorems}
\crefname{lemma}{Lemma}{Lemmas}
\crefname{lem}{Lemma}{Lemmas}
\crefname{remark}{Remark}{Remarks}
\crefname{prop}{Proposition}{Propositions}
\crefname{defn}{Definition}{Definitions}
\crefname{corollary}{Corollary}{Corollaries}
\crefname{cor}{Corollary}{Corollaries}
\crefname{section}{Section}{Sections}
\crefname{figure}{Figure}{Figures}
\crefname{quest}{Question}{Questions}

\newcommand{\N}{\mathbf{N}}
\newcommand{\Z}{\mathbf{Z}}

\newcommand{\R}{\mathbf{R}}

\def\sW{{\mathscr{W}}}

\def\ord1{{\sqsubset}}

\newcommand{\RN}[1]{%
  \textup{\uppercase\expandafter{\romannumeral#1}}%
}
\newcommand{\wml}[1]{{\operatorname{wml}}_{\ord1}^{#1}}
\newcommand{\sml}[1]{{\operatorname{sml}}_{\ord1}^{#1}}

\newtheorem{thm}{Theorem}[section]
\newtheorem{lemma}[thm]{Lemma}
\newtheorem{prop}[thm]{Proposition}
\newtheorem{cor}[thm]{Corollary}

\newtheorem{quest}[thm]{Question}

\theoremstyle{definition}
\newtheorem{definition}[thm]{Definition}
\newtheorem{example}[thm]{Example}
\newtheorem{remark}{Remark}

\newcommand{\pOrd}{\operatorname{p-Ord}}
\newcommand{\Ord}{\operatorname{Ord}}

\newcommand{\Prob}{\operatorname{Prob}}
\newcommand{\IRO}{\operatorname{IRO}}
\newcommand{\pIRO}{\operatorname{p-IRO}}
\newcommand{\IRS}{\operatorname{IRS}}
\newcommand{\stab}{\operatorname{stab}}
\newcommand{\Sub}{\operatorname{Sub}}

\newcommand{\Ext}{\operatorname{Ext}}

\newcommand{\Cay}{\operatorname{Cay}}
\newcommand{\supp}{\operatorname{Supp}}

\newcommand{\trivgp}{\langle e \rangle}
\newcommand{\G}{\Gamma}

\newcommand{\SL}{\operatorname{SL}}

\newcommand{\norm}[1]{\left\lVert #1 \right\rVert}
\newcommand{\abs}[1]{\left\lvert #1 \right\rvert}

\begin{document}
\subjclass[2020]{20F60; 37A15}
\keywords{Orders on groups, Random orders, Amenability}
\thanks{The first and second authors were supported by ISF grant 2919/19.}
\thanks{The third and second authors were supported by ISF grant  1058/18.}
	\begin{abstract}
		In this paper we study the action of a countable group $\Gamma$ on the space of orders on the group.  In particular, we are concerned with the invariant probability measures on this space, known as \emph{invariant random orders}. We show that for any countable group the space of random invariant orders is rich enough to contain an isomorphic copy of any free ergodic action, and characterize the non-free actions realizable. We prove a Glasner-Weiss dichotomy regarding the simplex of invariant random orders. We also show that the invariant partial order on $\SL_3(\Z)$ corresponding to the semigroup generated by the standard unipotents cannot be extended to an invariant random total order. We thus provide the first example for a partial order (deterministic or random) that cannot be randomly extended.
	\end{abstract}	
	\maketitle
\tableofcontents	
\section{Introduction}
The origins of the theory of orderable groups goes back to the  end of the nineteenth century and the beginning of the
twentieth century. It continues to be an active area of research, mainly due to its connections
with many different branches of mathematics. In this paper we extend a  fruitful  and relatively modern theme in this theory: The study of  the \emph{space of orders on a group} from a topological and a dynamical point of view. The space of left-invariant orders on $\Gamma$ is a zero-dimensional compact Hausdorff topological space  on which $\G$ acts by conjugation \cite{sikora2004topologyOrderings}. The study of this action from the point of view of topological dynamics proved to be a powerful tool, especially in the case where this action has no fixed points, i.e. the group is not bi-orderable. See for example \cite{LW:loc_ind}. For an exposition and further historical background on this point of view, see for instance \cite{N:10}. Inspired by the success of this theory, we turn to investigate the action of $\G$ by left-multiplication on the space of all total orders on $\G$. This is a much larger space, whose  fixed points, if such exist, are the left invariant orders. 

The central objects of this paper are 
\emph{invariant random orders}, namely probability measures on the space of orders  whose distribution is invariant with respect to multiplication (say, from the left). The term ``invariant random orders'' appeared in \cite{AMR:21}, and we refer the reader to \cite{Kieffer:SMB,stepin1978equidistribution} for earlier applications, in particular in the context of entropy theory for actions of amenable groups. There are further  recent applications of invariant random orders in the context of entropy theory  \cite{raimundo2021GibbsSofic,downarowicz2021Multiorders}.

Our results can be considered as evidence for the dynamical and geometric richness of  the space of invariant random orders.

 The organization of the paper is as follows:  In \cref{sec:general} we introduce some definitions and notation, as well as basic structural results on the space of orders on a group and the associated action.  In \cref{sec:SL3Z_nonextendable}, we show that for non-amenable groups it not always possible to extend a left-invariant partial order to a left-invariant random (total) order. This answers a question posed in \cite{AMR:21}, where extendability of partial invariant random orders in the amenable case was resolved. More specifically,  we demonstrate  non-extendability as above  with respect to the left-invariant partial order on the group $\SL_3(\Z)$  corresponding to the semi-group generated by the standard unipotent matrices. In \cref{sec:specification} we establish the \emph{ specification property} for the space of orders on a countable group $\Gamma$.  In \cref{sec:ord_ergodic_universality} we prove that any free ergodic action of $\Gamma$ can be ``realized'' as an invariant measure on the space of total orders. As for non-free ergodic actions, we provide a sufficient condition, as well as a necessary condition. These two are not so far from each other. In \cref{sec:simplex_IRO} we prove that for groups $\Gamma$ that do not admit property (T) the space of invariant random orders is the Poulsen simplex, thus establishing a Glasner-Weiss type dichotomy for the simplex of invariant random orders. We conclude with open questions and some additional remarks. 

 \begin{remark}
 Subsequently to the first arXiv version of this paper, building upon our results, Andrei Alpeev proved that \emph{any} non-amenable countable group admits a partial invariant order that cannot be extended to an invariant random total order, thus showing that the IRO extension property characterizes amenable groups \cite{alpeev2022IRO_ext}. 
 \end{remark}

 \begin{remark}
     This is a revision of the published of version of the paper uploaded in July 2026. In the published version, the proof of \Cref{prop:specification_on_amenable_is_poulsen} was incorrect. See also \Cref{rem:continuous_selection} regarding continuous and Borel selection. We thank Todor Tsankov for bringing the matter to our attention.
 \end{remark}

\section{Definitions, and basic observations} \label{sec:general}
\subsubsection*{Partial and total orders.} 
Let $\Gamma$ be a countable group. We denote by $\pOrd(\Gamma)$ the set of partial orders on $\Gamma$, namely the set of  binary relations on $\Gamma$ that are transitive, non-reflexive and antisymmetric. 
 \begin{itemize}
     \item Antisymmetry: $x \prec y$ implies that $y \not \prec x$
     \item Non-reflexivity: $x \not \prec x, \ \forall x \in \Gamma$
     \item Transitivity: $x \prec y$ and $y \prec z$ imply $x \prec z$
 \end{itemize}
 $\Ord(\Gamma) \subset \pOrd(\Gamma)$ will denote the subset of total orders, namely orders for which every two group elements are comparable as described below:
\begin{itemize}
     \item Total antisymmetry: For every $x \ne y \in \Gamma$ either $x \prec y$ or $y \prec x$.
 \end{itemize}
 Both collections admit a natural compact metrizable topology, upon identifying them as closed subsets of the set of all binary relations $\{0,1\}^{\Gamma \times \Gamma}$ endowed with the Tychonoff (product) topology.
 In the above, a partial or total $\prec$ is identified with its indicator function $R_{\prec}:\Gamma \times \Gamma \rightarrow \{0,1\}$ by $x\prec y$. 
 
The group $\Gamma \times \Gamma$ acts, by homeomorphisms, from the left, on $\Ord(\Gamma)$ by:
$$x \left[(\gamma,\delta) \cdot \prec \right] y \ \ {\text{ if and only if }} (\gamma^{-1} x \delta) \prec (\gamma^{-1} y \delta).$$ 
In this paper, 
 whenever we refer to the action of $\Gamma$ on $\Ord(\Gamma)$ it will be implicitly understood that $\Gamma$ is acting via its identification with 
$\Gamma \times \trivgp < \Gamma \times \Gamma$. We denote  the set of $\G$-fixed points of $\Ord(\G)$, equivalently the set of \emph{left-invariant orders} by
\[\Ord(\Gamma)^{\Gamma} := \Ord(\Gamma)^{\Gamma \times \trivgp}.\] 

If $\Ord(\G)^\G$ is nonempty we say that $\G$ is left-orderable. The left invariant orders are exactly the orders satisfying $x\prec y \iff \gamma x \prec \gamma y, \ \forall x,y,\gamma \in \Gamma$. A softer, more probabilistic notion of left invariance that will be the focus of this paper is the following: 
\begin{definition}
    A (left) {\it{invariant random order}} on $\Gamma$, or an $\IRO$ for short, is a $\Gamma$-invariant Borel probability measure on $\Ord(\Gamma)$. We will denote by $\IRO(\Gamma)$ the collection of all invariant random orders on $\Gamma$. Similarly $\pIRO(\Gamma)$ will denote the collection of {\it{invariant random partial orders}}, defined as $\Gamma$-invariant probability measures on $\pOrd(\Gamma)$. 
\end{definition}
Both $\IRO(\Gamma)$ and  $\pIRO(\Gamma)$, endowed with the $w^{*}$ topology become compact metrizable spaces.
While many groups do not admit a left invariant order, every countable group admits an IRO. Here is one construction to have in mind. Consider the $\Gamma$ equivariant map $\Phi:[0,1]^{\Gamma} \rightarrow \Ord(\Gamma)$ sending $\{\omega_{\gamma} \ | \ \gamma \in \Gamma\}$ to the order on $\Phi(\omega)$ defined by the requirement that $x \Phi(\omega) y$ if and only if $\omega_x < \omega_y$. If $\lambda$ denotes the Lebesgue (or any other atomless probability) measure on $[0,1]$ and $\Lambda = \lambda^\G$ the corresponding product measure on $[0,1]^{\Gamma}$, then $\Phi(\omega)$ is well defined for $\Lambda$-almost every $\omega$ and $\Phi_{*}(\Lambda)$ is an IRO on $\Gamma$.
The above random order, which is sometimes called ``the uniform random order'', is uniquely characterized by the property that for any finite set $F \subset \Gamma$ the restriction of $\prec$ to $F$ is uniformly distributed among the $|F|!$ possible permutations of $F$.
An early appearance of the uniform random order in the context of entropy theory is due to Kieffer \cite{Kieffer:SMB}, where it was used to prove an asymptotic equipartition theorem for amenable groups (see also \cite{stepin1978equidistribution} and the earlier paper \cite{pickel1971}).

\subsubsection*{Extension of orders.}
Given a partial order $\prec \in \pOrd(\Gamma)$ we denote by 
$$\Ext(\prec) = \{\ll \in \Ord(\Gamma) \ | \ x \prec y \Rightarrow x \ll y, \ \forall x,y \in \Gamma \}$$ the collection of total orders that extend the given partial order. $\Ext(\prec)$ is a closed subset of $\Ord(\Gamma)$ and is $\Gamma$-invariant whenever $\prec$ is. We refer to any $\ll \in \Ext(\prec)^{\Gamma}$ as an {\it{extension of $\prec$}} to a total invariant order. Again we will be interested in softer, more probabilistic, notions for extensions of orders. 

\begin{definition} \label{def:ext}
Let $\prec \in \pOrd(\Gamma)^{\Gamma}$.  $\mu \in \IRO(\Gamma)$ will be called a \emph{random extension} of $\prec$ if $\mu(\Ext(\prec))=1$. More generally if $\nu \in \pIRO(\Gamma)$, We will say that $\mu \in \IRO(\Gamma)$ extends $\nu$ if there exists a $\Gamma$-invariant probability measure $\theta$ on $\left(\Ord(\Gamma) \times \pOrd(\Gamma)\right)$, so that: $\theta$ projects onto $\mu$ and $\nu$ under the two projections and $\ll  \in \Ext(\prec)$ for $\theta$ almost every $(\ll,\prec) \in \Ord(\Gamma)\times \pOrd(\Gamma)$.
\end{definition}
Recall that such a $\Gamma$-invariant probability measure $\theta$ admitting $\mu$ and $\nu$ as marginals, is called {\it{a joining}} of $\mu$ and $\nu$. It is well known that if $\mu,\nu$ are both ergodic, then every ergodic component of $\theta$ is also a joining. So that in ergodic case $\theta$ can be taken to be ergodic without loss of generality. 

Deterministic random orders are ``determistically'' extendable within the class of  torsion-free locally nilpotent groups:
\begin{thm} (Rhemtulla \cite{Rh:72}, Formanek  \cite{For:73})
Every invariant partial order on a torsion-free locally nilpotent group can be extended to an invariant total  order. 
\end{thm}

It is well known that the conclusion of the above theorem fails if we relax the assumption of torsion-free locally nilpotent group for instance to the class of finitely generated torsion free solvable groups. For some examples and further references, see \cite{huang2016Extension}.
In contrast, extending invariant random partial  orders is possible under the much more general assumption of amenability. 
\begin{prop}[\cite{stepin1978equidistribution},\cite{AMR:21}]
Any invariant random partial order on an amenable group can be extended to an invariant random (total) order. 
\end{prop}
A proof of the above proposition follows by observing that the space of (not necessarily invariant) extensions of a given random order is a non-empty simplex on which the group acts. See \cite{AMR:21} for details. 

It seems natural to wonder if the amenability assumption above is necessary \cite[Question $2.2$]{AMR:21}. In the current paper, we will present a result showing that in general, extension of partial orders to IRO's is not possible, at least for some non-amenable groups. In a different direction, special kinds of partial orders can be randomly extended in any group. 



\begin{prop}
Let $\Gamma$ be a countable group and let $\Delta < \Gamma$ be a subgroup. Any IRO on $\Delta$ (viewed as an invariant random partial order on $\Gamma$) can be extended to an IRO on $\Gamma$. 
\end{prop}
\begin{proof}
The proof generalizes our prior construction of the uniform random order.  Let $\mu_0 \in \IRO(\Delta)$. Let $\Lambda = \lambda^{\Gamma/\Delta}$ be Lebesgue measure on $[0,1]^{\Gamma/\Delta}$ (namely the product of the Lebesgue measure taken in each coordinate), and let $X \subset [0,1]^{\Gamma/\Delta}$ denote the subspace of injective functions (so $x_{g \Delta} = x_{h \Delta}$ implies $g^{-1} h \in \Delta$ for all $x \in X$, $g,h  \in \Gamma$).
Then $\lambda^{\Gamma/\Delta}(X)=1$, so $\lambda^{\Gamma/\Delta}$ can be regarded as a probability measure on $X$.
We define a function $\Phi: \Ord(\Delta) \times X \to \Ord(\Gamma)$ as follows:
$\tilde \prec = \Phi(\prec,x)$ is given by
\[
g \tilde \prec h {\text{ if and only if }} 
 x_{g \Delta} < x_{h \Delta} \mbox{ or } 
(g \prec h \mbox{ and }  g^{-1}h \in \Delta) 
\]
\noindent $\tilde \prec \in \Ord(\Gamma)$ for every $\prec \in \Ord(\Delta)$ and $x \in X$. Also it is easy to verify that $\Phi$ is a $\Gamma$-equivariant map and that  $\tilde \prec$ extends $\prec$.

Now define $\tilde{\Phi}: \Ord(\Delta) \times X \to \Ord(\Delta) \times \Ord(\Gamma)$  by $\tilde{\Phi}(\prec,x) = (\prec, \Phi(\prec,x))$ and set $\theta = \tilde{\Phi}_{*}(\mu_0 \times \Lambda)$. All the properties mentioned in the end of the last paragraph show that $\theta$ is the joining needed in order to define and extension of $\prec$ to an IRO on $\Gamma$, according to \cref{def:ext}.
\end{proof}

\subsubsection*{Dynamical pasts and semigroups}

We now recall a simple and well known correspondence between left-invariant orders and semigroups not containing the identity, and observe that this correspondence can be meaningfully extended to a $\Gamma\times \Gamma$-equivariant bijective correspondence between $\pOrd(\Gamma)$ and a space we refer to as ``dynamical pasts''.

 A left invariant (partial) order on $\Gamma$ is uniquely determined by the semigroup of positive elements $\Phi_{<} := \{x \in \Gamma \ | \ e < \gamma \}$. Conversely, any semigroup $S$ in $\Gamma$ that does not contain the identity gives rise to a partial $\Gamma$-invariant order $<_S$, given by $x <_S y \iff x^{-1}y\in S$. The above correspondences define a bijection between left invariant orders on $\Gamma$ and semi-groups of $\Gamma$ that do not contain the identity. A semigroup $S$ not containing the identity  corresponds to a left invariant total order if and only if it has the additional property that  $\Gamma = S \sqcup \{e\} \sqcup S^{-1}$. Such semigroups  are known as {\it{algebraic pasts}}. A semigroup $S$ not containing the identity  corresponds to a bi-invariant order if and only if it the group $\Gamma$ normalizes $S$ in the sense that $g s g^{-1} \in S$ for every $s \in S$ and $g \in \G$. 
 
 There is a natural bijection $\Psi:\{0,1\}^{\Gamma \times \Gamma} \to (\{0,1\}^\Gamma)^\Gamma$ between the space $\{0,1\}^{\Gamma \times \Gamma}$ of binary relations on $\Gamma$ and the space $(\{0,1\}^\Gamma)^\Gamma$
 of functions from $\Gamma$ to the space $\{0,1\}^\Gamma$, which we naturally identify as the space of functions from $\Gamma$ to the space of subsets of $\Gamma$. This bijection is futhermore a homeomorphism (where the topology on $\{0,1\}^{\Gamma \times \Gamma}$ is the product of $\G \times \G$ copies of $\{0,1\}$ with the discrete topology  and the topology on $(\{0,1\}^\Gamma)^\Gamma$ is the ``iterated'' product topology). 
 For $\prec \in \pOrd(\Gamma)$ the image $\Psi_\prec$ can be written as follows:
 \[
 \Psi_{\prec}(x) =  \left\{y \in \Gamma \ | \ x \prec  y \right\},\ x\in \Gamma.
 \]
 For $\prec \in \pOrd(\Gamma)$ and $x \in \Gamma$, the set $\Psi_{\prec}(x) \subset \Gamma$ can be thought of as \emph{the past of $x$} with respect to $\prec$.
 On the space $(\{0,1\}^\Gamma)^\Gamma$ we have a natural self-homeomorphism $\phi \mapsto \tilde \phi$, given by
 \[ \tilde\phi(x) := x^{-1}\phi(x)= \{ x^{-1}y \in \Gamma \ |  y \in \phi(x)\}.\]
 
 The composition of $\Psi$ and the self-homeomorphism above gives another bijection $\Phi$ between $\{0,1\}^{\Gamma \times \Gamma}$ and  $(\{0,1\}^\Gamma)^\Gamma$, namely,  $\Phi:\{0,1\}^{\Gamma \times \Gamma} \to (\{0,1\}^\Gamma)^\Gamma$.
 For $\prec \in \pOrd(\Gamma)$ we can write:
 \[\Phi_{\prec}(x)  =  \left\{\gamma \in \Gamma \ | \ x \prec x \gamma\right\}, \ x\in \Gamma.
 \]
 For $x \in \Gamma$ and $\prec \in \pOrd(\Gamma)$, the set $\Phi_{\prec}(x) \subset \Gamma$ can be thought of as ``the directions pointing to the past from $x$''. 
 
 A  function $S:\Gamma \rightarrow \{0,1\}^\Gamma$  is in the image of $\pOrd(\Gamma)$ under $\Phi$ if and only if it satisfies the following conditions:

\begin{itemize}
    \item Antisymmetry: For every $x,\gamma \in \Gamma$ with $e \ne \gamma$ at most one of the conditions $\gamma \in S(x)$, $\gamma^{-1} \in S(x \gamma)$ can hold. 
    \item Non-reflexivity: $e \not \in S(x), \ \forall x \in \Gamma$
    \item Transitivity: $\gamma S(x\gamma) \subset S(x), \ \forall \gamma \in S(x)$. 
    \end{itemize}
A function  $S:\Gamma \rightarrow \{0,1\}^\Gamma$  in the image of $\Ord(\Gamma)$  under $\Phi$ satisfies the following property in addition:
\begin{itemize}
    \item Total antisymmetry: For every $x,\gamma \in \Gamma$ with $e \ne \gamma$ exactly one of the conditions $\gamma \in S(x)$ or $\gamma^{-1} \in S(x\gamma)$ holds. 
\end{itemize}

For a left invariant (total) order  $\prec$, $\Phi_\prec(x)$ is independent of $x$, and is precisely the semigroup (algebraic past) corresponding to $x$. 

\begin{example}\label{example:Z_2_ext}
Let $P = \Z_+^n \setminus \{0\} \subset \Z^n$ denote the semigroup corresponding to the positive ortant in $\Z^n$ with zero removed. Then $P$ defines an invariant partial order on $\Z^n$, that we denote by $\ord1$.
We have
\begin{equation} \label{eqn:eprec}
\Ext(\ord1) = \{\ll \in \Ord(\Z^n) \ | \ P \subset \Phi_{\ll}(\overline{x}), \ \forall \overline{x} \in \Z^n\}.
\end{equation}
Given a function $u:\Z^n \to \mathbb{R}$ which is injective and strictly monotone with respect to $\ord1$ in the sense that $u(x) < u(y)$ whenever 
$y-x \in P$, we can define an element $\prec_u \in  \Ext(\ord1)$ by $x \prec_u y$ iff $u(x) < u(y)$. In fact, this is true in general: For any countable group $\Gamma$, any injective function $u:\Gamma \to \mathbb{R}$ defines a total order $\prec_u \in \Ord(\Gamma)$  as above. Given $\ord1 \in \pOrd(\Gamma)$ we have that $\prec_u \in \Ext(\ord1)$ if and only if $u:\Gamma \to \mathbb{R}$ is $\ord1$-monotone. Conversely, any $\prec \in \Ext(\ord1)$ is of the form $\prec = \prec_u$ for some injective, $\ord1$-monotone function. Composing $u$ from the left with a strictly increasing function from $\mathbb{R}$ to $\mathbb{R}$ does not change the resulting order $\prec_u$. In the case $\Gamma = \Z^n$ we can assume without loss of generality that the function $u:\Z^n \to \mathbb{R}$ is the restriction of some continuous function $\tilde U:\mathbb{R}^n \to \mathbb{R}$ (or even $1$-Lipschitz, piecewise linear, smooth and so on) having the property that each level set of $\tilde u$ intersects $\Z^n$ in at most $1$ point. The level lines of $\tilde u$  (at least in the piece-wise smooth case), which  are $n-1$-dimensional surfaces, uniquely determine $\tilde u$. At least in the case $\Gamma=\Z^n$ this point of view provides a method of visualizing elements of $\Ext(\ord1)$.
For instance, \cref{fig:cylord} shows the level lines of a function corresponding to some element of $\Ext(\ord1)$, restricted to some bounded square region in  $\mathbb{R}^2$, where the larger bold point represents the zero vector, and the smaller bold points represent other elements of $\Z^2$.
\begin{figure}[ht]
    \centering
    \includegraphics[width=5cm]{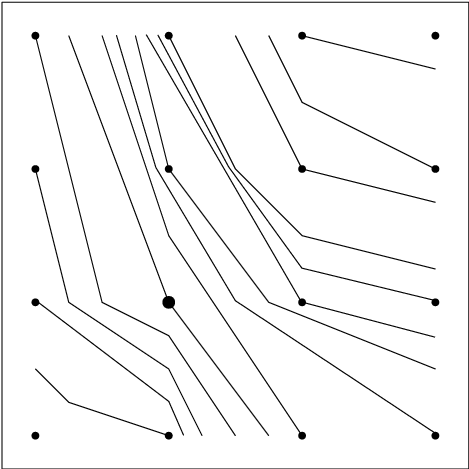}
    \caption{The level lines of a function corresponding to an element of  $\Ext(\ord1)$, restricted to a finite window.}
    \label{fig:cylord}
\end{figure}

We briefly  recall that  the standard and well-known classification of invariant  total orders on $\Z^n$ and identify which of these are elements of $\Ext(\ord1)$: To any invariant total order on $\Z^n$   there is an associated $(n-1)$-dimensional linear subspace of $\mathbb{R}^n$. An $(n-1)$-dimensional linear subspace of $\mathbb{R}^n$ which does not contain any non-zero integral points corresponds to exactly two invariant total orders,  whose corresponding algebraic pasts are the two half-planes in the complement of $V$, intersected with $\Z^d$. Such an element corresponds to $\Ext(\ord1)$ if and only if the corresponding half-plane cotains $P$.
The algebraic pasts corresponding to an $(n-1)$-dimensional linear subspace $V$ of $\mathbb{R}^n$ which contains a  subgroup $\Gamma_0$ of $\Z^n$ of rank $k$ for some $k \le n -1$, are exactly the union of an algebraic past of $\Gamma \cong \Z^k$ (which we can identify inductively) and one of the half-planes in the complement of $V$.

In the case of an invariant total order $\prec \in \Ext(\ord1)$ corresponding to an irrational subspace $V$, the ``level surfaces'' could be taken as affine spaces parallel to $V$. In the case of a subspace $V$ containing non-zero rational points, one has to perturb the level surfaces slightly so that each level surface contains at most one integral point. 
\end{example}

\begin{example}\label{example_heisenberg}
The discrete Heisenberg group is given by:
\[H =  \left\{ \left. [a,b,c]:=\begin{pmatrix} 1&a&c\\0&1&b\\0&0&1\end{pmatrix} \right| a,b,c \in \Z\right\}.\]
The discrete Heisenberg group $H$ is generated by the three matrices
\begin{equation}
\begin{array}{ccc}
x=\begin{pmatrix}
	1 & 1 & 0\\
	0 & 1 & 0\\
	0 & 0 & 1
\end{pmatrix} & 
z=\begin{pmatrix}
1 & 0 & 1\\
0 & 1 & 0\\
0 & 0 & 1
\end{pmatrix} & 
y=\begin{pmatrix}
	1 & 0 & 0\\
	0 & 1 & 1\\
	0 & 0 & 1
\end{pmatrix}
\end{array}
\end{equation}
The matrix $z$ is the commutator of $x$ and $y$, and it commutes with each of them.  These relations give a presentation of the  discrete Heisenberg group:
\[H = \left \langle x,y,z \ | \ [x,z]=[y,z]=[x,y]z^{-1} \right \rangle. \]
The following well known formula can be verified by a straightforward induction: 
\begin{equation} \label{eqn:hcomm}
    [x^k,y^l]=z^{kl}, \quad \forall k,l \in \Z
\end{equation} 
The vertices of the Cayley graph $C = \Cay(H,\{x,y,z\})$ of $H$ are naturally identified with $\Z^3 \subset \R^3$. 
Let $P \subset H$ denote the semigroup of matrices in $H$ whose entries are all non-negative, exluding the identity matrix. Then $P$ is precisely the semigroup generated by $x,y,z$. Let $\ord1$ denote the partial left-invariant order corresponding to $P$.
For any $[A,B,C] \in H$, and $\prec \in \Ext(\ord1)$ we know that $\Psi_{\prec}([A,B,C])$ is bound to contain $[A,B,C]P$, which assumes the form of a skewed cone
\[[A,B,C] P = \{[A+a,B+b,C+c+Ab] \ | \ a,b,c \in \Z_{\ge 0} \} \subset \Psi_{\prec}([A,B,C]).\]
Examples of such sheared cones, including $P$ itself at the origin, appear in \cref{fig:cones}. 
\begin{figure}[ht]
    \centering
    \includegraphics[width=9cm]{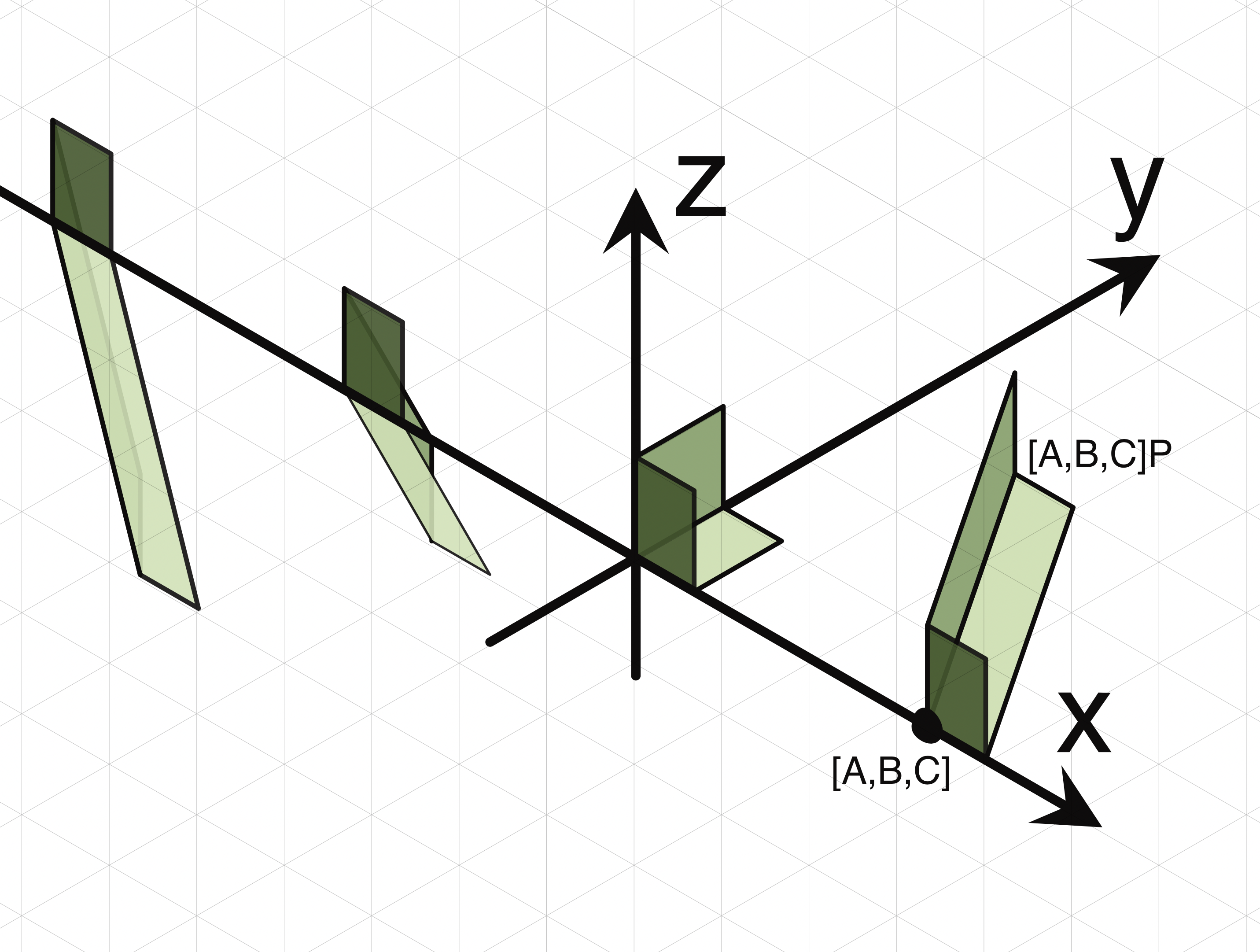}
    \caption{Sets of the form $\Psi_{\ord1}(g) = \{h \in H \ | \ g \ord1 h \}$, assume the form of sheared octants in the Heisenberg group.}
    \label{fig:cones}
\end{figure}
\end{example}

\section{A non-extendable partial invariant order on $\SL_3(\Z)$} \label{sec:SL3Z_nonextendable}

Denote by $\ord1$ the left-invariant partial order on $\Gamma = \SL_3(\Z)$ whose semigroup of positive elements is generated as a semigroup by the standard unipotent matrices  
$\Psi_{\ord1} = \langle a_1,a_2,\ldots, a_6\rangle$ where
\begin{equation}
\label{eqn:unip}
\begin{array}{ccc}
a_1=\begin{pmatrix}
	1 & 1 & 0\\
	0 & 1 & 0\\
	0 & 0 & 1
\end{pmatrix} & 
a_2=\begin{pmatrix}
1 & 0 & 1\\
0 & 1 & 0\\
0 & 0 & 1
\end{pmatrix} & 
a_3=\begin{pmatrix}
	1 & 0 & 0\\
	0 & 1 & 1\\
	0 & 0 & 1
\end{pmatrix}\\
a_4 = 
\begin{pmatrix}
	1 & 0 & 0\\
	1 & 1 & 0\\
	0 & 0 & 1
\end{pmatrix} & 
a_5=
\begin{pmatrix}
	1 & 0 & 0\\
	0 & 1 & 0\\
	1 & 0 & 1
\end{pmatrix} & a_6=
\begin{pmatrix}
	1 & 0 & 0\\
	0 & 1 & 0\\
	0 & 1 & 1
\end{pmatrix}
\end{array}
\end{equation}
	In this section we prove the following:
	\begin{thm} \label{thm:non_ext}
		There does not exist an invariant random order on $\Gamma:=\SL_3(\Z)$ with the property that any $e \ne A \in \SL_3(\Z)$ without negative entries is almost surely positive.
	\end{thm}
This shows that without the amenability assumption, it is not in general possible to extend a partial invariant order to a total invariant order, providing a negative answer to a question posed in \cite{AMR:21}. The proof is an adaptation of Dave Morris' proof that any finite index subgroup of $\SL_n(\Z)$ with $n \ge 3$ is non-orderable \cite[Proposition $3.3$]{Wit:94} (see also the monograph \cite{deroin2016groups}, or the survey \cite{morris:survay}). 

\begin{remark} A slightly larger semigroup is the semigroup consisiting of all matrices $e \ne A$ which are entrywise positive. Geometrically  the partial order defined by this larger semigroup is given by $g < h$ if and only if $g \ne h$ and $hP \subset gP$ where $P \subset \R^3$ is the positive octant. Of course our theorem above implies that this order too, cannot be extended to to an IRO on $\Gamma$. We thank Andrei Alpeev for drawing our attention to a mistake in an early version of this paper, where we didn't clearly distinguish between these two semigrops. 
\end{remark}

A straightforward verification shows that $[a_{i},a_{i+1}]=e$ and $[a_{i-1},a_{i+1}]=a_i$, with all subscripts read modulo $6$. Our goal is to show that $\Ext(\ord1)$ admits no $\Gamma$-invariant Borel probability measure. 

As in \cite{Wit:94} we will be interested in the question when an element $a \in \Gamma$ is much larger than another element $b$. Here are four natural definitions capturing this notion in some special cases:  

\begin{definition}
Let $a,b \in \Gamma$ such that $e \ord1 a,b$, and $\langle a,b \rangle = \Z^2$.
 \begin{eqnarray*}
\wml{+}(a,b) & := & \{ \prec \in \Ext(\ord1) \ | \ \forall M>0, \exists N>0 {\text{ such that }} a^{-k}b^{Mk} \prec e, \ \forall k\ge N\}, \\
\wml{-}(a,b) & := & \{ \prec \in \Ext(\ord1) \ | \ \forall M>0, \exists N>0 {\text{ such that }} e \prec b^{-Mk}a^{k}, \ \forall k\ge N\} \\
\sml{+}(a,b) & := &  \left\{ \prec \in \Ext(\ord1) \ | \ \exists q > 0 {\text{ such that }}  a^{-q}b^{n} \prec e  \; \forall n\in \N \right\} \\
\sml{-}(a,b) & := &  \left\{ \prec \in \Ext(\ord1) \ | \ \exists q>0 {\text{ such that }} e \prec b^{-n}a^{q}  \; \forall n\in \N \right\}, 
\end{eqnarray*}
The acronyms wml and sml stand for weakly and strongly much larger, respectively.
\end{definition}
These definitions are tailored for our current proof. Similar definitions make sense in much more general settings: $\Gamma$ could be a general countable group, $\ord1$ any $\Gamma$-invariant partial order. Note that this definition looks only at the dynamic past $\Phi_{\prec}(e)$ at the identity.

That $a$ be much larger than $b$ should entail inside the $a,b$-plane that the line $U$ separating the positive and negative elements has to come very close to the $b$-axis. The stronger notion above, requires $U$ to be at a bounded distance from the $b$-axis. The weak notion requires $U$ to be eventually closer to the $b$-axis than any linear line with a finite slope. The $\pm$ superscript represents whether this closeness is measured along the positive and negative directions of the $b$-axis. It is quite possible for an order to exhibit completely different behavior in these two regions. Indeed if $P$ is the positive qudrant without zero, then every order $\prec \in \Ext(\ord1)$ with the property that $P \subset \Psi_{\prec}(\overline{0}) \subset P-v$ for some positive $v \in \Z^2$ will in fact satisfy $\prec \in \sml{+}(a,b) \cap \sml{+}(b,a)$. So that this latter set is far from being empty. For invariant orders all four notions above coincide and contain exactly one order in $\Ext(\ord1)$ - the lexicographic order. 

\begin{lemma}\label{lem:A_B_C_rel_2}
Let $a,b \in \Gamma$ be such that $e \ord1 a,b$ and $\langle a,b \rangle = \Z^2$. Then 
	\begin{enumerate}
	\item \label{lem:a1} $\sml{\pm}(a,b) \subseteq \wml{\pm}(a,b)$. 
	\item \label{lem:a2} $\sml{-}(a,b) \subseteq \wml{+}(b,a)^c$.
\end{enumerate}
\end{lemma}
\begin{proof}
Let $a,b \in \Gamma$ be as given. 
\begin{enumerate}
        \item Suppose $\prec \in \sml{+}(a,b)$ so that $a^{-q}b^n \prec e$ for some $q$ and all $n$. Then for any $0 < M$ and $k > q$, we have $a^{-k}b^{Mk} \prec a^{-k}b^{Mk}a^{k-q} = a^{-q}b^{Mk} \prec e$, as required. The other statement follows similarly.
        \item Suppose, by way of contradiction that $\prec \in \sml{-}(a,b) \cap \wml{+}(b,a)$. In particular $\prec \in \sml{-}(a,b)$ yields $q>0$ such that $e \prec a^qb^{-n}$ for all $n$, while $\prec \in \wml{+}(b,a)$ ensures that $b^{-n}a^n \prec e$ for all large enough $n$. By our hypotheses that $a$ and $b$ commute and $a$ is positive, we obtain $e \prec a^qb^{-n} \prec a^nb^{-n} \prec e$, a contradiction to transitivity and non-reflexivity.
\end{enumerate}
\end{proof}

The following lemma is a slight elaboration of \cite[Lemma $3.2$]{Wit:94}. 
\begin{lemma}\label{lem:A_B_C_rel_1} 
With $a_i \in \SL_3(\Z)$ the basic unipotent matrices defined in \cref{eqn:unip}, and all indexes taken modulo $6$ we have:  
	\begin{enumerate}
		\item \label{lem:b1} $\wml{+}(a_i,a_{i+1})^c \subseteq \sml{-}(a_{i+2},a_{i+1})$.
		\item \label{lem:b2} $\wml{+}(a_i,a_{i+1}) \subseteq \sml{+}(a_{i-1},a_i)$.
	\end{enumerate}
\end{lemma}
Both statements  actually reflect some information about the geometry of the discrete Heisenberg group $H = \langle a_i,a_{i+1},a_{i+2}\rangle$. Before proceeding to the formal proof let us explain this geometric point of view:

As in \cref{example_heisenberg} we denote Heisenberg matrices by triplets of integers
$$[a,b,c] \cong \begin{pmatrix} 1 & a & c \\ 0 & 1 & b \\ 0 & 0 & 1 \end{pmatrix}$$ Note that when we restrict ourselves to the $\langle x,z \rangle$ and $\langle z,y \rangle$ planes the group operations are well represented by the arithmetic in $\R^3$. 
\begin{figure}[ht]
    \centering
    \includegraphics[width=8cm]{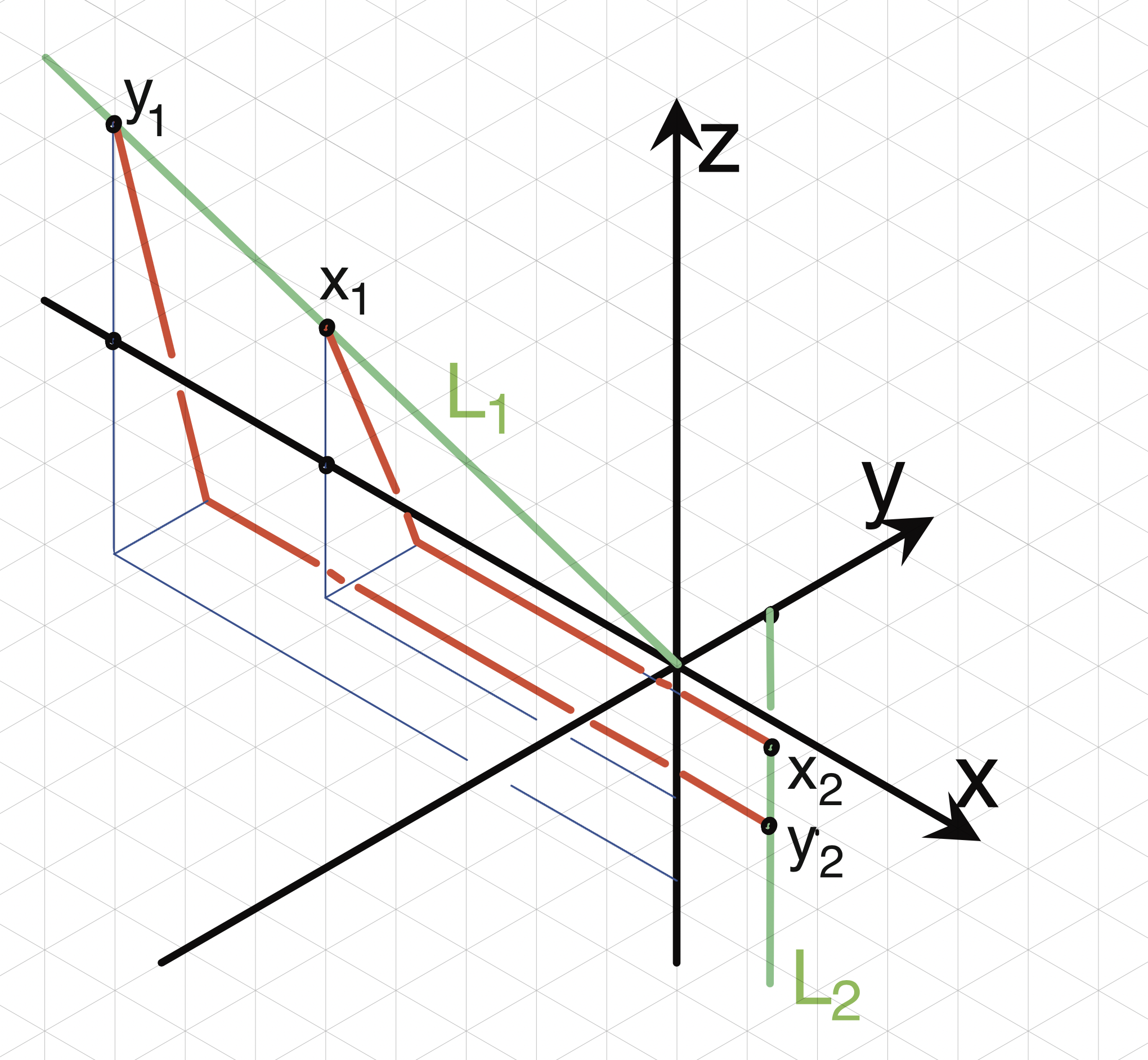}
    \caption{An illustration of $\wml{+}(x,z)^c \subseteq \sml{-}(y,z)$} 
    \label{fig:hies}
\end{figure}
The proof  of the first claim is depicted in \cref{fig:hies}, where $a_i,a_{i+1},a_{i+2}$ are represented by the $x,z,y$ coordinates respectively. Suppose $\prec \in \wml{+}( a_i,a_{i+1})^c$. Then there exists $M \in \N$ such that $e \prec a_i^{-k}a_{i+1}^{Mk}$ for infinitely many  values of $k$. Graphically this yields the green line $L_1$, of slope $-M$ in the $x,z$ plane, with the property that infinitely many integral points falling on it are $\prec$-positive. Two such points $x_1,y_1$ are depicted in the picture. As in \cref{fig:cones}, together with each such positive point comes a sheared cone of points that are bound to be even bigger according to $\ord1$. Following the red lines, along the boundary of these sheared cones, we obtain another line $L_2$ parallel to the negative direction of the $z$ axis in the $y,z$-plane, containing infinitely many positive points. Since we are extending $\ord1$, all points directly above each positive point on $L_2$ is also positive, so that $L_2$ consists entirely of positive points, which is exactly what we need. Keeping the geometric interpretation in mind, we proceed to prove the statements algebraically:
\begin{proof}
	    As above suppose $\prec \in \wml{+}( a_i,a_{i+1})^c$. Then there exists $M \in \N$ such that $e \prec a_i^{-k}a_{i+1}^{Mk}$ for infinitely many  values of $k$.  Because $\prec \in \Ext(\ord1)$ and $e\ord1 a_{i+2},a_{i}$ for any $q \in \N$ we have $a_i^{-k}a_{i+1}^{Mk} \prec a_i^{-k}a_{i+1}^{Mk} a_{i+2}^q a_i^{k}$. Using the commutation relations in the Heisenberg group (cf. (\cref{eqn:hcomm}) at the end of \cref{sec:general}) we obtain
	    $a_i^{-k}a_{i+1}^{Mk} a_{i+2}^q a_i^{k} = a_{i+1}^{Mk-qk}a_{i+2}^q$. 
	    Choose $q=2M$ with $M$ as above.
	    It follows that there are infinitely many $k$'s such that
	   $e \prec a_{i+1}^{-Mk}a_{i+2}^q$.
	    By multiplying by positive powers of $a_{i+1}$ we conclude that  $e \prec a_{i+1}^{-n}a_{i+2}^q$ for all $n$, so $\prec \in \sml{-}(a_{i+2},a_{i+1})$ as required. This completes the proof of (\cref{lem:b1}). 
	    
	    Suppose now that $\prec \in \wml{+}(a_i,a_{i+1})$. Choosing $M=1$ in the definition of that set, we can find $N \in \N$ such that 
	    $a_i^{-n}a_{i+1}^n \prec e$ for all $n >N$. Multiplying from the right by negative powers of  $a_{i-1}$ and $a_{i+1}$, we have
	    $a_{i}^{-n}a_{i+1}^na_{i-1}^{-2}a_{i+1}^{-n} \prec e.$ Using (\cref{eqn:hcomm}) again, we have $a_{i-1}^{-2}a_{i+1}^{-n}=a_{i}^{2n}a_{i+1}^{-n}a_{i-1}^{-2}$.
	    so that 
	    \[a_{i-1}^{-2}a_i^{n}=
	    a_i^{-n}a_{i+1}^na_{i}^{2n}a_{i+1}^{-n}a_{i-1}^{-2}=
	    a_i^{-n}a_{i+1}^na_{i-1}^{-2}a_{i+1}^{-n} = \prec e\] 
	    Multiplying by negative powers of $a_i$ if needed, we conclude that  $a_{i-1}^{-2}a_i^n \prec e$ for all $n$. So $\prec \in \sml{+}(a_{i-1},a_i)$ which completes the proof of (\cref{lem:b2}).
\end{proof}

Let $\sml{-} := \bigcap_{i=1}^6 \sml{-}(a_i,a_{i-1})$ and $\sml{+} := \bigcap_{i=1}^6 \sml{+}(a_i,a_{i+1})$, where all the indices are considered modulo $6$. 
\begin{lemma}\label{lem:Ord_0_partition}
	\[\Ext(\ord1) = \sml{+} \cup \sml{-}.\]
\end{lemma}
\begin{proof}
	Take any $\prec \in \Ext(\ord1)$.
	
	Case 1: suppose $\prec \in \wml{+}(a_1,a_2)^c$.  By \cref{lem:A_B_C_rel_1} (\cref{lem:b1}), $\prec \in \sml{-}(a_3,a_2)$.  Then by \cref{lem:A_B_C_rel_2} (\cref{lem:a2}), $\prec \in \wml{+}(a_2,a_3)^c$.  Repeat this argument five more times to obtain $\prec \in \cap_{i=1}^6 \sml{-}(a_i,a_{i-1})$.
	
	Case 2: suppose $\prec \in \wml{+}(a_1,a_2)$. By \cref{lem:A_B_C_rel_1} (\cref{lem:b2}), $\prec \in \sml{+}(a_6,a_1)$.  Then by \cref{lem:A_B_C_rel_2} (\cref{lem:a1}), $\prec \in \wml{+}(a_6,a_1)$. Repeat this argument five more times to obtain $\prec \in \cap_{i=1}^6 \sml{+}(a_i,a_{i+1})$.
\end{proof}
 The above lemma is an analogue of Dave Witte Morris' proof of \cite[Proposition $3.3$]{Wit:94}. In the deterministic setting, it is rather immediate that neither $\sml{+}$ nor  $\sml{-}$ contain an invariant order. In our random setting more work must be done to show neither support an invariant probability measure. It is interesting to note that all our results so far, and in particular \cref{lem:A_B_C_rel_1} and \cref{lem:Ord_0_partition} are deterministic in the sense that no probability is involved. 
 
 To show that $\sml{+},\sml{-}$cannot support an invariant probability measure, we will eventually decompose them into  \emph{wandering sets}: Suppose a countable group $\Gamma$ acts on a Borel space $X$ by Borel automorphisms.
We say $A \subseteq X$ is \emph{wandering} if there exists $g \in \G$ such that $(g^nA)_{n\in \Z}$ are all pairwise disjoint.  Let $\sW$ be the collection of all countable unions of wandering sets. We say that $A$ is \emph{non-recurrent} with respect to $g \in \Gamma$ if for all $x \in X$, $\sum_{n =1}^\infty 1_A(g^n(x)) < +\infty$.  

The following simple lemma collects some basic facts about wandering and non-recurrent sets. The arguments are all elementary set-theoretic, no assumptions at all are made about the algebraic structure of  the countable group $\Gamma$.
\begin{lemma}\label{lem:wander}
The following are true:
\begin{enumerate}
    \item $\sW$ is a $\G$-invariant $\sigma$-ideal. That is, it is closed under countable unions and taking subsets.
    \item Any $A \in \sW$ has zero measure with respect to any $\G$-invariant probability measure.
    \item If $A \subset X$ is non-recurrent with respect to any $g \in \Gamma$ then $A \in \sW$.
    \item If $A \subset X$ is such that there exists $g \in \G$ and $N \in \N$ such that for all $n \geq N$, $g^nA \cap A = \emptyset$, then $A \in \sW$.
    \item If $A \subset X$ is such that there exists $g\in\G$ and $N \in \N$ such that for all $n \geq N$, $g^nA \cap A \in \sW$, then $A \in \sW$.
    
\end{enumerate}
\end{lemma}

\begin{proof}

    (1). It is straightforward that for every $A \in \sW$ and $B \subset A$, $B \in \sW$;  and that a countable union of elements of $\sW$ is also in $\sW$. To see that $\sW$ is $\Gamma$-invariant, observe that if  $B$ is wandering with respect to $g$  and $h \in \G$, then $h(B)$ is wandering with respect to $hgh^{-1}$.
    
    (2). For any $\G$-invariant probability measure $\mu$, if $A \in \sW$ has positive measure then there exists $B \subset A$ with $\mu(B)>0$ and $g \in \G$ with $(g^nB)_{n \in \Z}$ pairwise disjoint, so $\mu$ has infinite measure, a contradiction.
    
    (3). Let $A$ be non-recurrent with respect to $g$. For $k \in \Z$ let $A_k = \{x \in X: g^kx \in A, \forall l>k\ g^lx \notin A\}$.  Notice that $(A_k)_{k\in \Z}$ are pairwise disjoint, and $A = \sqcup_{k \geq 0} (A_k \cap A)$.  Also notice that for $n,k \in \Z$, $g^n A_k \subset A_{k-n}$, so $A_k \cap A$ is wandering, so $A \in \sW$.
    
    (4). Note that in this case for all $x \in X$,  $\sum_{n=1}^\infty 1_A(g^n(x)) \leq N+1$, so $A$ is non-recurrent.  The result follows from (3).
    
    (5).  For $n \geq N$ let $A_n = g^n A \cap A$, and let $A_0 = A\setminus (\cup_{n \geq N} A_n)$.  Each $A_n \in \sW$ by assumption, and we claim that $A_0$ is also in $\sW$.  Indeed note that for $n \geq N$, $g^nA_0 \cap A_0 \subset g^nA \cap A$, but also $g^n A_0 \cap A_0 \subset A_0$, so by (4) $A_0 \in \sW$.  So $A \in \sW$ being a countable union of sets in $\sW$.
    
\end{proof}

\begin{lemma}\label{lem:non_recurr}
If $\Gamma$ acts on $X$ and $A \subset X$ satisfies that there exists $N >0$, $k \in \N$ and $g_1,\dots,g_k \in\Gamma$ such that for any $n_1,\ldots,n_k \ge N$ we have $\bigcap_{i=1}^k g_i^{n_i}A = \emptyset$ then $A \in \sW$.
\end{lemma}
\begin{proof}
We prove the following slightly stronger statement by induction on $k\in \N$:
Suppose there exists $g_1,\ldots,g_k \in \Gamma$ and $N \in \N$ such that for every $n_1,\ldots,n_k \ge N$ 
$A \cap \left( \bigcap_{i=1}^k g_i^{n_i}A \right) \in \sW$. Then $A \in \sW$.
The case $k=1$ follows directly from  \cref{lem:wander}(5). 
Now suppose there exists $g_1,\ldots,g_k,g_{k+1} \in \Gamma$ and $N \in \N$ such that for every $n_1,\ldots,n_k,n_{k+1} \ge N$ 
$A \cap \left( \bigcap_{i=1}^{k+1} g_i^{n_i}A \right) \in \sW$. For $n_1,\ldots,n_k \ge N$ let
$A_{n_1,\ldots,n_k}= A \cap \left( \bigcap_{i=1}^{k} g_i^{n_i}A \right)$. Then 
for all $n_{k+1} \ge N$ we have that $A_{n_1,\ldots,n_k} \cap g^{n_{k+1}}A_{n_1,\ldots,n_k} \subseteq A \cap \left( \bigcap_{i=1}^{k+1} g_i^{n_i}A \right) \in \sW$. By  \cref{lem:wander} (1) and (5) it follows that $A_{n_1,\ldots,n_k} \in \sW$. Since this holds for all $n_1,\ldots,n_k \ge N$, it follows that $A \in\sW$ by the induction hypothesis.
\end{proof}

\begin{lemma}\label{lem:W_weakly_dissipative}
	The sets $\sml{+}$ and $\sml{-}$ are in $\sW$.
\end{lemma}
\begin{proof}
	We will prove that  $\sml{-} \in \sW$, as the proof that $\sml{+} \in \sW$ is completely analogous. For $q \in \N$ and $1 \le u \le 6$, denote  $B(q,i) = \{\prec \in \Ext(\ord1) \ | \ e \prec a_i^q a_{i-1}^{-n}, \ \forall n \in \N\}$. Observe that \[\sml{-} = \bigcup_{q \in \N} \left(\bigcap_{i=1}^6 B(q,i) \right).\]  
	Since $\sW$ is closed under countable unions, it suffices to show that $\bigcap_{i=1}^6 B(q,i) \in \sW$ for every $q \in \N$. By \cref{lem:non_recurr} is suffices to show that for any $n_1,...,n_6 >q$ we have $\bigcap_{i=1}^6 a^{-n_i}_{i-1} B(q_i,i) = \emptyset$.
	
	Indeed,  
	for every $m >q$ and $\prec \in a_{i-1}^{-m} B(q,i)$ we have \[a_{i-1}^q \prec a_{i-1}^{m} \prec a_{i-1}^m a_i^q a_{i-1}^{-m}=a_i^q.\] 
	In particular, if $n_1,\ldots,n_6 > q$ and $\prec \in \bigcap_{i=1}^6 a^{-n_i}_{i-1} B(q,i)$, it follows that 
	$a_1^q \prec a_2^{q} \prec ... \prec a_6^{q} \prec a_1^{q}$. 
	
	
\end{proof}
\begin{prop}\label{prop:Ord_SL_3_weakly_dissipative}
	The set $\Ext(\ord1) \in \sW$, and in particular   
 there is no invariant random total order on $\Gamma$ extending $\ord1$.  
\end{prop}

\begin{proof}
That $\Ext(\ord1) \in \sW$ follows directly from \cref{lem:Ord_0_partition} together with  \cref{lem:W_weakly_dissipative}.	The ``in particular'' part follows by \cref{lem:wander} (2).
\end{proof}  

\section{The specification property of $\Ord(\Gamma)$}\label{sec:specification}

\begin{definition}
A topological dynamical system $\Gamma \curvearrowright X$ has the \emph{ specification property} if for every $\epsilon >0$ there exists a non-empty finite  subset $F \subset \Gamma$ such that for every $x_1,x_2 \in X$ and any $K \subset \Gamma$ there exists $x \in X$ such that 
\[d(g(x),g(x_1)) \le \epsilon \mbox{ for all } g \in K,\]
and 
\[d(g(x),g(x_2)) \le \epsilon \mbox{ for all } g \in \Gamma \setminus (FK).\]

\end{definition}

\begin{remark}
The term {\it{specification property}} is not used consistently throughout the literature. Some manuscripts refer to this property as {\it{uniform specification}}, {\it{strong specification}} or as {\it{strong irreducibly}}. Other manuscripts yet, use the term specification for slightly modified notions. 
\end{remark}
\begin{remark}
It is routine to check that the specification property is  independent of the particular metric $d$. If $X$ is a totally disconnected compact space then $\Gamma \curvearrowright X$ has the specification property if and only if for every partition of $P$ of $X$ into clopen sets there exists a finite set $F \subset \Gamma$ such that for every $x_1,x_2 \in X$ and any $K \subset \Gamma$ there exists $x \in X$ such that
\[P(g(x))=P(g(x_1))  \mbox{ for all } g \in K,\]
and
\[P(g(x))=P(g(x_2)) \mbox{ for all } g \in \Gamma \setminus (FK),\]
where $P(x)$ is the partition element of $P$ containing $x$.
\end{remark}

\begin{remark}\label{rem:strongly_irreducible}
If $A$ is a finite set and $X \subseteq A^{\Gamma}$ is a $\Gamma$-subshift then it is straightforward that the specification property of $\Gamma \curvearrowright X$ is equivalent to $X$ being \emph{strongly irreducible}: There exists a finite subset $F \subset \Gamma$ such that for any $x_1,x_2 \in X$ and any $K \subset \G$ there exists $x \in X$ such that $x\mid_K = x_1\mid_{K}$
and $x\mid_{\Gamma \setminus (KF)} = x_2\mid_{\Gamma \setminus (KF)}$. 
\end{remark}

\begin{remark}\label{rem:specification_cantor_strongly_irreducible}
If $X$ is a totally disconnected compact metrizable space, than $\G \curvearrowright X$ has the specification property if and only if any subshift factor of  $\G \curvearrowright X$ is strongly irreducible.
\end{remark}

\begin{definition}
 Given $D \subseteq \Gamma$ and $\prec \in \Ord(\G)$, denote 
 \[[\prec]_D := \left\{ \tilde \prec \in \Ord(\G):~ \tilde \prec\mid_{D \times D} = \prec\mid_{D \times D}
 \right\}.\]
\end{definition}
Whenever $D \subset \Gamma$ is finite and $\prec \in \Ord(\Gamma)$, we have that $[\prec]_D$ is a clopen subset of $\Ord(\Gamma)$. The sets of the form $[\prec]_D$ for $\prec \in \Ord(\G)$ and $D$ a finite subset of $\G$ are called \emph{cylinder sets}. For instance, for $\G= \Z^2$, \cref{fig:cylord} describes the collection ``level lines'' for elements of a specific cylinder set in  $\Ord(\Gamma)$.

\begin{prop}\label{prop:Ord_Gamma_sepcification}
The action $\Gamma \curvearrowright \Ord(\Gamma)$ has the specification property.
\end{prop}
\begin{proof}
It suffices to show that for any finite $D \subset \G$ there exists a finite $F\subset \G$ such that for any $\prec_1,\prec_2 \in \Ord(\G)$ and $K \subset \Gamma$ there exists $\prec \in \Ord(\G)$ 
that ``$D$-shadows $\prec_1$ on $K$'' and ``$D$-shadows $\prec_2$ on $\G \setminus (FK)$'' in the following sense:
\[[g(\prec)]_D=[g(\prec_1)]_D  \mbox{ for all } g \in K,\]
and
\[[g(\prec)]_D=[g(\prec_2)]_D \mbox{ for all }\mbox{ for all } g \in \Gamma \setminus (FK).\]
Indeed, given a finite subset $D \subset \Gamma$, let $F = DD^{-1}$.  Then for any $K \subset \Gamma$  and $\prec_1,\prec_2 \in \Ord(\G)$ one can define $\prec \in \Ord(\G)$ by declaring that $g_1 \prec g_2$ if and only if one of the following holds:
\begin{enumerate}
    \item $g_1,g_2 \in K^{-1}D$ and $g_1 \prec_1 g_2$.
    \item $g_1,g_2 \not\in K^{-1}D$ and $g_1 \prec_2 g_2$.
    \item $g_1 \in K^{-1}D$ and $g_2 \in \Gamma \setminus K^{-1}D$.
\end{enumerate}
Note that the condition that $\prec$ $D$-shadows $\prec_1$ on $K$ is implied by $[\prec]_{K^{-1}D}=[\prec_2]_{K^{-1}D}$, which is given by (1).  Similarly the condition that $\prec$ $D$-shadows $\prec_2$ on $\G\setminus (FK)$ is implied by $[\prec]_{(\G\setminus (FK))^{-1}D}=[\prec_1]_{(\G\setminus (FK))^{-1}D}$.  Now we claim to have chosen $F$ so that $(\G\setminus (FK))^{-1}D \cap K^{-1}D = \emptyset$, which is equivalent to $\G\setminus (K^{-1}F^{-1}) \cap K^{-1}DD^{-1} = \emptyset$, which is true since $F = DD^{-1} = F^{-1}$.  It follows that (2) implies $\prec$ $D$-shadows $\prec_2$ on $\G\setminus (FK)$, and $(3)$ ensures that $\prec$ is total. 
\end{proof}

\begin{remark}\label{remark:borel_specification_ord_gamma}
The proof of \cref{prop:Ord_Gamma_sepcification} actually shows an a priori stronger property for $\G \curvearrowright  \Ord(\G)$: For any finite subset $D \subset \G$ there exists $F \subset \G$ and an equivariant \footnote{The equivariance  is with respect to the action of   $\Gamma$  on subsets given by  $g.K = Kg^{-1}$.}  Borel function $\Phi:\Ord(\G) \times \Ord(\G) \times 2^\Gamma \to \Ord(\G)$  such that $\Phi(\prec_1,\prec_2,K)$ $D$-shadows $\prec_1$ on $K$ and $D$-shadows $\prec_2$ on $\G \setminus (FK)$. We will refer to this property as the  \emph{Borel equivariant specification property}.  In fact in the proof of \cref{prop:Ord_Gamma_sepcification} the map $\Phi$ is continuous. We do not know if in general this property follows automatically from the ``usual'' specification property.
\end{remark}

\section{Ergodic universality of $\Ord(\Gamma)$: Realizing ergodic systems via IROs}\label{sec:ord_ergodic_universality}

The purpose of this section is to provide a characterization of ergodic $\Gamma$-actions that can be ``realized as IROs'', in the ergodic theoretic sense. This turns out to be a rather general class, as it includes all essentially free ergodic actions. 
We briefly recall the notion of an invariant random subgroup, which turns out to be key in the characterization of ergodic $\Gamma$-actions that can be realized as IROs.
Let $\Sub(\G)$ be the set of all subgroups of $\G$. The space $\Sub(\G)$ can be naturally viewed as a closed, hence compact subset of $\{0,1\}^\G$. The group $\G$ acts on $\Sub(\G)$ by conjugation.  
Recall that a probability measure on $\Sub(\Gamma)$ that is invariant with respect to the action of $\Gamma$ by conjugation is called an \emph{invariant random subgroup} \cite{AbertGlasnerVirag2014}, abbreviated by IRS.
Let $\IRS(\G) \subset \Prob(\Sub(\G))$ denote the space of invariant random subgroups for $\G$.
For any $\G$-space $X$, and any $x \in X$,  we denote by $\stab(x) \in \Gamma(x)$  the stabilizer subgroup of $x$. 
A basic observation in  \cite{AbertGlasnerVirag2014} is that any p.m.p action $\Gamma \curvearrowright (X,\mathcal{B},\mu)$, gives rise to a invariant random subgroup   via the stabilizer map $\stab: X\to \Sub(\G)$. That is, $\stab_{*}(\mu) \in \IRS(\Gamma)$. Conversely, it was shown in \cite{AbertGlasnerVirag2014} that any invariant random subgroup is realizable as the stabilizer of some p.m.p action. 
The invariant random subgroups arising  from  invariant random orders are subject to an obvious orderability constraint:
\begin{prop}\label{prop:stab_orderable}
For any $\prec \in \Ord(\G)$, the subgroup $\stab(\prec) < \G$ is left-orderable.
\end{prop}
\begin{proof}
Fix $\prec \in \Ord(\G)$.  Let $\G_0 = \stab(\prec)$. By definition, for any $g \in \G_0$, $g(\prec) = \prec$.  In particular, $g(\prec|_{\G_0\times \G_0}) = \prec|_{\G_0\times \G_0}$ - the restriction of $\prec$ to $\G_0$ is also $g$-invariant.  Thus $\prec|_{\G_0\times \G_0}$ is a left-invariant order on $\G_0$.     
\end{proof}

\begin{cor}
		For any invariant random order on $\Gamma$ the stabilizer is almost surely an orderable subgroup. 
		Namely, if  $\mu \in \IRO(\Gamma)$ then 
		\[\mu \left( \prec \in \Ord(\Gamma) :~ \stab(\prec) \mbox{ is left-orderable }\right)=1.\]	
\end{cor}

\begin{thm}\label{thm:universality}
	Let $\Gamma$ be  a countable group , and  $\Gamma \curvearrowright (X,\mathcal{B},\mu)$ an ergodic probability-preserving $\Gamma$-action not supported on a finite set. Then $\Gamma \curvearrowright (X,\mathcal{B},\mu)$ is measure-theoretically isomorphic to a $\Gamma$-invariant measure on $\Ord(\Gamma)$ if and only if 
	there exists an equivariant measurable function $\pi:X \to \pOrd(\Gamma)$ such that for almost every $x \in X$, the restriction of $\pi(x) \in \pOrd(\Gamma)$ to the stabilizer $\stab(x) < \Gamma$ is a left-invariant total order on $\stab(x)$.
\end{thm}

\begin{proof}

One direction is trivial:
Suppose $\G\curvearrowright (X,\mathcal{B},\mu)$ is measure isomorphic to an IRO on $\G$ via $\Phi: X \to \Ord(\G)$.  Let $x\in X$. Since $\Ord(\G) \subseteq \pOrd(\G)$, we can take $\pi=\Phi$, then clearly the restriction of $\pi(x)$ to $\stab(x)=\stab(\Phi(x))$ is a total order on $\stab(x)=\stab(\Phi(x))$, which is furthermore invariant (as explained in the proof of \cref{prop:stab_orderable}).

	Let $\Gamma \curvearrowright (X,\mathcal{B},\mu)$ be an ergodic probability  preserving $\Gamma$-action,
	and suppose $\pi:X \to \pOrd(\Gamma)$ is an equivariant measurable function such that for almost every $x \in X$, the restriction of $\pi(x) \in \pOrd(\Gamma)$ to the stabilizer $\stab(x) < \Gamma$ is an left-invariant total order on $\stab(x)$.
	By ergodicity, since  $\mu$ is not supported on a finite orbit, it has no atoms,
	%
	%
	so $(X,\mathcal{B},\mu)$ is a standard Lebesgue space with a non-atomic probability measure. We can assume  that $X=[0,1]$, that $\mathcal{B}$ is the Lebesgue $\sigma$-algebra , and that $\mu$ is Lebesque measure. 
	Given $x \in X$,
	denote by $\tilde \prec_x \in \Ord(\stab(x))$ the restriction of $\pi(x) \in \pOrd(\Gamma)$ to $\stab(x)$.
	Define $\prec_x \in \Ord(\Gamma)$ by $g \prec_x h$ iff $g(x) < h(x)$ or $g(x)=h(x)$ and $e \tilde \prec_x  g^{-1}h$. Because $g(x)=h(x)$ if and only if $ g^{-1}h \in \stab(x)$, it follows that indeed $\prec_x \in   \Ord(\Gamma)$. The map $\Phi:X \to \Ord(\Gamma)$ given by  $\Phi(x):= \prec_x$ is clearly Borel and equivariant. It remains to check that  it is injective on a set of full measure. We do so by explicitly describing a Borel inverse.
	
	We recall the notion of a \emph{pointwise ergodic sequence of measures} for a group $\Gamma$. A sequence of probability measures $(\nu_n)_{n=1}^\infty$ on $\Gamma$ is called \emph{pointwise ergodic for $\Gamma$} if for any probability preserving $\Gamma$-action $\Gamma \curvearrowright (X,\mathcal{B},\mu)$ and any $f \in L^1(X,\mathcal{B},\mu)$ we have $\int f(g(x)) \nu_n(g) \to \int f d\mu$ for $\mu$-almost every $x\in X$.  Any countable group $\Gamma$ admits a pointwise ergodic sequence of measures \cite{kakutani1951random,oseledec65}, for instance the convolution powers of a symmetric probability measure on $\Gamma$ whose support generates $\Gamma$.  For background and a historical account see for instance \cite{BN:13}. 
	Let $(\nu_n)_{n=1}^\infty$  be a pointwise ergodic sequence of measures on $\Gamma$.
	We claim that almost surely we have
	\[ x = \lim_{n \to \infty} \nu_n(\left\{h\in \Gamma:~ h \prec_x e\}\right).\]
	
	Indeed, this is equivalent to showing that for   every $t \in [0,1] \cap \mathbb{Q}$, and almost every $x \in [0,t] $ we have 
	\[
	\lim_{n \to \infty} \nu_n\left(\left\{h\in \Gamma:~ h \prec_x e\right\}\right) \le t.\]
	
	But by definition $h \prec_x e$ implies that $h(x) \le x$. It follows that for any $x \in [0,t]$
	\[ \nu_n\left(\left\{h\in \Gamma:~ h \prec_x e\right\}\right)  \le \nu_n\left(\left\{h\in \Gamma:~ h(x) \le  t \right\}\right). \]
	
	By pointwise ergodicity of $\nu_n$, we have that for almost every $x$,
	\[ \lim_{n \to \infty} \nu_n\left(\left\{h\in \Gamma:~ h(x) \le  t \right\}\right) = \mu([0,t]) = t.\]
	This shows that the equivariant  Borel map $\Psi:\Ord(\Gamma) \to [0,1]$ given by
	\[\Psi(\prec )= \liminf_{n \to \infty} \nu_n\left(\left\{h\in \Gamma:~ h \prec e\right\}\right) \]
	Satisfies $\Psi(\Phi(x))=x$ for almost every $x$.
\end{proof}
In particular we have: 
\begin{cor}\label{cor:free_realization}
For any countable group $\Gamma$, any essentially free, ergodic probability preserving $\Gamma$-action can be realized as an IRO.
\end{cor}

\begin{proof}

If  $\Gamma \curvearrowright (X,\mathcal{B},\mu)$ is ergodic and essentially free, the map $\pi:X \to \pOrd(\Gamma)$ defined by letting $\pi(x)$ be the identity relation,  satisfies the required property that the restriction of $\pi(x)$ to the trivial subgroup is an invariant total order. The fact that the action is essentially free means that $\stab(x)$ is almost surely equal to the trivial subgroup.
\end{proof}

\begin{cor}\label{orederable_realization}
If $\Gamma$ is left-orderable, any infinite ergodic action of $\Gamma$ on a space can be realized as an IRO.
\end{cor}
\begin{proof}
Suppose $\prec_0 \in \Ord(\Gamma)$ is left-invariant order. The map  $\pi:X \to \pOrd(\Gamma)$ defined by $\pi(x):=\prec_0$  satisfies the required property that the restriction of $\pi(x) \in \pOrd(\Gamma)$ to $\stab(x)$ is an invariant total order.
\end{proof}

\begin{remark}
An inspection of the proof of \cref{thm:universality} shows that the uniform total order is measure-theoretically isomorphic to the Bernoulli shift $[0,1]^\Gamma$, equipped with Haar measure.
\end{remark}

\begin{remark}
As a particular consequence of \cref{orederable_realization} we see that that
whenever $\G$ is a countable sofic group, then the action of $\G$ on $\Ord(\G)$ has infinite sofic topological entropy with respect to any sofic approximation sequence.
\end{remark}

\begin{remark}\label{rem:ergodicity_univer}
The ergodicity assumption in \cref{thm:universality} is necessary.
For instance,
the group $\Gamma = \Z$ has only $2$ 
invariant deterministic total orders so the trivial action on a space with more than $2$ points cannot be realized as an IRO.
\end{remark}

\begin{remark}
In the case $\G=\Z^d$, it follows from  
 \cite{CM:2021} that any action of $\Z^d$ having the specification property and infinite topological entropy can realize any essentially free measure preserving action as an invariant measure. Since the action of $\G$ on $\Ord(\G)$ has the specification property and infinite topological entropy, this shows that any essentially free probability-preserving action of $\Z^d$ can be realized as an IRO on $\Z^d$, so the ergodicity assumption in \cref{thm:universality} can be removed in this case. It is plausible  that the arguments of \cite{CM:2021} can be extended to more general amenable groups.
\end{remark}

\section{The structure of the simplex $\IRO(\G)$}
\label{sec:simplex_IRO}
Let us consider the space of all invariant random orders on $\G$ as a compact convex set. One can ask what properties of the group $\G$ can be detected by considering $\IRO(\G)$ as a metrizable Choquet simplex, up to affine homeomorphisms. A \emph{Poulsen simplex} is a metrizable Choquet simplex whose exteme points are dense. Lindenstrauss, Olsen and Sternfeld \cite{LOS:poulsen} have shown that there is a unique  Poulsen simplex up to affine homemomorphism. A \emph{Bauer simplex} is a metrizable Choquet simplex whose exteme points are closed.
In \cite{GW:97} Glasner and Weiss proved the following striking dichotomy for the simplex of invariant measures of the shift action $\G \curvearrowright K^\G$, for a compact set $K$: If $\G$ has property $(T)$, then $\Prob_{\Gamma}(K^{\Gamma})$ is a Bauer simplex.
If $\G$ does not have property $(T)$, then $\Prob_{\Gamma}(K^{\Gamma})$ is a Poulsen simplex.

We now show a similar dichotomy holds for $\IRO(\G)$:

\begin{thm}\label{thm:IRO_Poulsen_Bauer}
Let $\G$ be a countable group. If $\G$ has property $(T)$, then $\IRO(\G)$ is a Bauer simplex.
If $\G$ does not have property $(T)$, then $\IRO(\G)$ is a Poulsen simplex.
\end{thm}

\begin{proof}
If  $\G$ has property $(T)$ it follows directly form \cite[Theorem 1]{GW:97} that $\IRO(\G)$ is a Bauer simplex. Now suppose $\G$ does not have property $(T)$. In order to prove that $\IRO(\G)$ is a Poulsen simplex,  we will show that for any two ergodic IROs  $\nu_1,\nu_2 \in \IRO(\G)$, 
the average $\frac{1}{2}(\nu_1+\nu_2)$ can be ``well weak-$*$ approximated'' by an ergodic IRO.
Specifically, it suffices to show that for any ergodic $\nu_1,\nu_2 \in \IRO(\G)$, any $\epsilon >0$ , any finite $D \subset \G$  there exists an ergodic $\nu \in \IRO(\G)$ such that 
\[
 \nu([\prec]_D) - \frac{1}{2}\left( \nu_1([\prec]_D) + \nu_2([\prec]_D)\right) < \epsilon\  ~\forall \prec \in \Ord(\G).
\]
So fix $\epsilon$ and  a finite $D \subset \G$.
By Glasner-Weiss \cite{GW:97} , there exists a weakly mixing $\G$ action on some space $(X,\mu)$ with asymptotically invariant sets. In particular there exists a measurable set $A \subset X$ with $\mu(A)=\frac{1}{2}$ such that $\mu(\bigcap_{g \in DD^{-1}}g(A)) > \frac{1}{2}-\frac{1}{2}\epsilon$ and $\mu(\bigcap_{g \in DD^{-1}}g (A^c)) > \frac{1}{2}-\frac{1}{2}\epsilon$. 
Let $\Phi:\Ord(\Gamma) \times \Ord(\Gamma)\times 2^\Gamma \to \Ord(\Gamma)$ be a Borel equivariant function witnessing the Borel equivariant specification property for $\G \curvearrowright \Ord(\G)$ as in \cref{remark:borel_specification_ord_gamma} and \cref{prop:Ord_Gamma_sepcification}. Let $\lambda$ be an ergodic joining of $\nu_1$ and $\nu_2$.  Let $\varphi: X \to 2^\G$ be the $\G$-equivariant function such that $\varphi(x)_e = 1_A(x)$. Then $\lambda\times \varphi_*\mu$ is an ergodic measure on $\Ord(\G) \times \Ord(\G) \times 2^\G$. Let $\nu$ be the push-forward of $\lambda\times \varphi_*\mu$  via $\Phi$. Then for every $\prec \in \Ord(\G)$ 
\begin{multline}
\nu_1([\prec]_D) \cdot \left(1- \mu( \bigcap_{g \in DD^{-1}}gA^c)\right) + \nu_2([\prec]_D) \cdot (1-\mu(A)) + \mu\bigg(X \setminus \Big(A^c \cup \big(\bigcap_{g \in DD^{-1}}gA^c\big)^c\Big)\bigg) \\
\ge \nu([\prec]_D) \ge
\nu_1([\prec]_D) \cdot \mu(A) + \nu_2([\prec]_D) \cdot \mu( \bigcap_{g \in DD^{-1}}gA^c).
\end{multline}

The leftmost-hand-side is bounded from above by $\frac{1}{2}\left( \nu_1([\prec]_D) + \nu_2([\prec]_D)\right) + \epsilon$ and the righthand-most-side is bounded from below by $\frac{1}{2}\left( \nu_1([\prec]_D) + \nu_2([\prec]_D)\right) - \epsilon$ which completes the proof. 
\end{proof}

\begin{remark}\label{rem:Borel_spec_T}
The proof of \cref{thm:IRO_Poulsen_Bauer} actually shows that for any countable group $\Gamma$ that does not have property $(T)$, for any action $\G \curvearrowright X$ with the Borel equivariant specification property the simplex of invariant measures is a Poulsen simplex. 
\end{remark}
The above remark motivates us to ask the following question:
\begin{quest}
Let $\Gamma$ be a countable group that does not have property $(T)$, and let $\G \curvearrowright X$ be an action with the specification property. Is the  simplex of invariant probability measures a Poulsen simplex?
\end{quest}
In view of  \Cref{rem:Borel_spec_T},
an affirmative answer to the above question would follow if it is the case that the specification property implies the Borel equivariant specification property (see \Cref{remark:borel_specification_ord_gamma}).
In this direction we have the following partial result:

\begin{prop}\label{prop:specification_on_amenable_is_poulsen}
If $\Gamma$ is a countably infinite amenable group, $O$ is compact metrizable, and $\G \curvearrowright O$ has the specification property, then the simplex of $\Gamma$-invariant probability measures on $O$ is a Poulsen simplex.
\end{prop}

\begin{proof}
Suppose $\Gamma$ is an amenable group acting on a compact metrizable space $O$. We fix once and for all a metric $d$ on $O$. Assume that $\G \curvearrowright O$ has the specification property. As in the proof of  \cref{thm:IRO_Poulsen_Bauer}, it suffices to show that for any ergodic $\nu_1,\nu_2 \in \Prob_\Gamma(O)$, any $w^*$-neighborhood $U \subseteq \Prob(O)$ of $\frac{1}{2}(\nu_1 + \nu_2)$ contains an ergodic element of $\Prob_\Gamma(O)$. Fix $\lambda \in \Prob^{\Gamma}(O \times O)$ to be an ergodic joining of $\nu_1$ and $\nu_2$. 

For the sake of definiteness we fix our attention to a basic open neighborhood. Namely fix a finite collection of continuous functions $Q \subset C(O)$ and some $\epsilon > 0$ and set from now on
$$U = U(Q,\epsilon) = \left\{\nu \in \Prob^{\Gamma}(O) \ | \ \left| \int f d\nu - \frac{1}{2} \int f  d(\nu_1 + \nu_2) \right| < \epsilon, \quad \forall f \in Q \right\}.$$
As a continuous functions on a compact space are bounded and uniformly continuous we fix $M = \max\left\{1,\sup \{\norm{f}_{\infty} \ | \ f \in Q\}\right\}$ and $\delta>0$ such that $\abs{f(o')-f(o'')} < \frac{\epsilon}{6}$ whenever $d(o',o'')<\delta$ for every $f \in Q$. 

Let \(F \subset \Gamma\) be a specification with tolerance
\(\delta/2\). After possibly enlarging $F$ we may assume $e \in F$. Given any $o_1,o_2 \in O$, $y \in \{0,1\}^{\Gamma}$, and applying the specification property to the set \(
K_0(y)=\{\gamma\in\Gamma:\ y(\gamma^{-1}\eta)=0
\text{ for every }\eta\in F\}
\), yields some $o\in O$ such that:
\begin{equation} \label{eqn:interpolate}
\begin{array}{ll}
d(\gamma o,\gamma o_1) \le \delta/2 & {\text{ whenever }} y(\gamma^{-1}\eta) = 0, \ \forall \eta\in F \\
d(\gamma o,\gamma o_2) \le \delta/2 &{\text{ when }} y(\gamma^{-1}\eta) = 1, \  \forall\eta \in F.
\end{array}
\end{equation} 
We will say that such an $o\in O$, ``$\delta/2$-interpolates $o_1$ and $o_2$ according to $y$''.  

Appealing to Glasner--Weiss \cite{GW:97}\footnote{This step only requires that $\G$ does not have property $(T)$.} choose a weakly mixing
compact metrizable \(\Gamma\)-system \((X,\mu)\) and a clopen set
\(A\subset X\), with \(\mu(A)=1/2\), such that, for the equivariant map
\[
\varphi:X\to\{0,1\}^{\Gamma},
\qquad
\varphi(x)_\gamma=1_A(\gamma^{-1}x),
\]
the sets
\[
B_1=\{x:\varphi(x)_\eta=0\text{ for all }\eta\in F\},
\qquad
B_2=\{x:\varphi(x)_\eta=1\text{ for all }\eta\in F\},
\]
and
\[
B_0=X\setminus(B_1\cup B_2)
\]
satisfy
\begin{equation} \label{eqn:Biest}
\left|\mu(B_i)-\frac12\right|< \frac{\epsilon}{12M} \
{\text{for }} i=1,2, \ {\text{and,}} \qquad
\mu(B_0) \leq \frac{\epsilon}{12M}.
\end{equation}

Let $Z = X \times O \times O \times O$, with $\pi^{1,2,3}:Z \to X \times O \times O$, and $\pi^4:Z \to O$ denoting the respective projections on the first three and on the fourth coordinates. Let $\varphi: X \rightarrow \{0,1\}^{\Gamma}$ be the $\Gamma$-equivariant function such that $\varphi(x)_{\gamma} = 1_{A}(\gamma^{-1} x)$. We also denote $\overline{B}_i = B_i \times O \times O \times O \subset Z$ for $i \in \{0,1,2\}$. Note that $Z = \overline{B}_0 \sqcup \overline{B}_1 \sqcup \overline{B}_2$ forms a partition of $Z$. 

Let $P = P_{\mu,\lambda,\delta}$ denote the set of all $\theta \in \Prob(Z)$ that satisfy the following two properties: 
\begin{itemize}
\item $\pi^{1,2,3}_*(\theta)=\mu \times \lambda$, 
\item for $\theta$-almost every $(x,o_1,o_2,o) \in Z$, 
$o$ $\delta/2$-interpolates $o_1$ and $o_2$ according to $\varphi(x)$.
\end{itemize}
Clearly $P$ is a $\Gamma$-invariant compact convex subset of $\Prob(Z)$. We will conclude by establishing the following properties which, together, yield the desired ergodic measure in $U$:
\begin{enumerate}[label=\alph*)]
\item \label{itm:face} $P^{\Gamma}$ is a nonempty face of $\Prob^{\Gamma}(Z)$.
\item \label{itm:image} $\pi^4_* (P^{\Gamma}) \subset U$. 
\end{enumerate}

We first show that $P^{\Gamma}$, if nonempty, is a face of $\Prob^{\Gamma}(Z)$. Assume therefore that $\theta = \beta \theta_1 + (1-\beta) \theta_2$
with $\theta \in P^{\Gamma}$, $\theta_1,\theta_2 \in \Prob^{\Gamma}(Z)$ and $\beta \in [0,1]$. Since $\mu$ is weakly mixing and $\lambda$ is ergodic we know that $\mu \times \lambda$ too is ergodic. It follows directly that $\pi^{1,2,3}(\theta_i) = \mu \times \lambda$ for both $i \in \{1,2\}$.
Let $R \subseteq Z$ be the set of points $(x,o_1,0_2,0)$ such that $o$ $\delta/2$-interpolates $o_1$ and $o_2$ according to $\varphi(x)$. By the assumption that $\theta \in O$ it follows that $\theta(R)=0$ it follows that $\beta \theta_1(R)+(1-\beta)\theta_2(R)=1$, so $\theta_1(R)=\theta_2(R)=1$ unless $\beta \in \{0,1\}$.

To establish \ref{itm:face}, it remains to show that $P^{\Gamma} \ne \emptyset$. Since $\Gamma$ is amenable and $P$ is a convex compact $\Gamma$-space it would be enough to show that $P \ne \emptyset$.  
Let $O^*$ denote the space of closed subsets of $O$ (with the Fell topology, induced by the Hausdorff metric). Consider the map $\Phi_0: X \times O \times O  \to O^*$ defined by 
 \[
 \Phi_0(x,o_1,o_2) = \left\{ o \in O:~ o ~ \frac{\delta}{2} \mbox{-interpolates } o_1 \mbox{ and } o_2 \mbox{ according to } \varphi(x)\right\},\]
 for $o_1,o_2 \in O$ and $x \in X$, the specification property implies that $\emptyset \ne \Phi_0(x,o_1,o_2)$ is non-empty for every $o_1,o_2 \in O$ and $x \in X$. It can be directly verified that $\Phi_0(x,o_1,o_2)$ is a closed subset of $O$. Furthermore, $\Phi_0:X \times O \times O  \to O^*$ is a clearly a Borel function. By the Kuratowski--Ryll-Nardzewski measurable selection theorem, there exists a Borel map
\(\Psi:O^*\to O\) such that \(\Psi(A)\in A\) for every non-empty closed set \(A\subseteq O\). Let $\Phi:X \times O \times O \to O$ be given by $\Phi = \Psi \circ \Phi_0$, this is a Borel map. Note that $\Phi$ is not equivariant because $\Psi$ is not generally equivariant. Nevertheless the push-forward of \(\mu\times\lambda\) under
\[
(x,o_1,o_2)\mapsto (x,o_1,o_2,\Phi(x,o_1,o_2))
\]
belongs to \(P\). This yields our desired measure, which completes the proof of \ref{itm:face}.

To prove \ref{itm:image} fix $\theta \in P^\Gamma$ and set $\nu = \pi^4_*(\theta)$ and fix some $f \in Q$. Then
\begin{equation} \label{eqn:split}
\int_O f(o) d\nu(o) = \int_Z f(o) d\theta(x,o_1,o_2,o) = \int_{\overline{B}_0} f(o) d\theta + \int_{\overline{B}_1} f(o) d\theta +\int_{\overline{B}_2} f(o) d\theta, 
\end{equation}
By Equation (\ref{eqn:Biest}) 
\begin{equation} \label{eqn:B0est}
\abs{\int_{\overline{B}_0}f(o)d\theta} < \frac{\epsilon}{6}.
\end{equation} 
Since \(\pi^{1,2,3}_*\theta=\mu\times\lambda\), and \(B_i\) depends only
on the first coordinate,
\[
\int_{\overline B_i} f(o_i)\,d\theta
=
\mu(B_i)\int_O f\,d\nu_i.
\]
On \(\overline B_i\), the interpolation condition gives
\(d(o,o_i)\leq\delta/2\), and therefore
\[
|f(o)-f(o_i)|<\frac{\epsilon}{6}.
\]
Thus
\begin{equation} 
\begin{aligned} \label{eqn:B12est}
\left|
\int_{\overline B_i} f(o)\,d\theta
-
\frac12\int_O f\,d\nu_i
\right|
&\leq
\int_{\overline B_i}|f(o)-f(o_i)|\,d\theta
+
\left|\mu(B_i)-\frac12\right|
\left|\int_O f\,d\nu_i\right|  \\
&<
\frac{\epsilon}{6}+\alpha M
\leq
\frac{\epsilon}{6}+\frac{\epsilon}{12}
<
\frac{\epsilon}{3}.
\end{aligned}
\end{equation}

Now \ref{itm:image} follows directly from Equations (\ref{eqn:split}), (\ref{eqn:B0est}) and (\ref{eqn:B12est}).
To conclude the proof take $\theta$ to be an extreme point of $P^{\Gamma}$. By \ref{itm:face} $P^{\Gamma}$ is a face of $\Prob^{\Gamma}(Z)$ so $\theta$ is in fact an ergodic $\Gamma$-invariant measure on $Z$. By \ref{itm:image} $\pi^4_*(\theta) = \nu$ is the desired ergodic measure in $U = U(Q,\epsilon)$.

\end{proof}

\begin{remark}\label{rem:continuous_selection}
In general, there does not exist a \emph{continuous} selection function $\Psi:O^* \to O$ such that $\Psi(A) \in A$ for every non-empty $A \in O^*$ (for instance take $O=\R/\Z$ to be the $1$-dimensional torus, and restrict to pairs of antipodal points, namely sets of the form $A=\{t,t+\frac{1}{2}\}$.
In a previous version of this paper it was falsely claimed that continuous selection functions exist (if $O$ can be topologically embedded  in the unit interval, then the function $\Psi(A):=\min(A)$ does the job).
\end{remark}

\section{Further discussion and open questions}\label{sec:open_questions}
We conclude with a discussion of some further directions, questions and related problems.

\subsection{Extension of random orders}
A countable group $\G$ has the \emph{IRO-extension property} if every partial invariant random order on $\G$ can be extended to a random invariant (total) order. As mentioned earlier, amenable groups have the IRO-extension property. In \cref{sec:SL3Z_nonextendable} we showed that  $\SL_3(\Z)$ does not have the IRO-extension property, providing a first example for a countable group for which this property fails. 
\begin{quest}\label{quest:IRO-extension}
Does there exist a non-amenable group with the IRO-extension property? 
\end{quest}
Following the first arXiv version of our paper, \Cref{quest:IRO-extension} was solved negatively by Andrei Alpeev \cite{alpeev2022IRO_ext}. Thus showing that the IRO-extension property is equivalent to amenability for countable groups. It is interesting to note that Alpeev uses our construction as one of the building blocks for his theorem.

\subsection{Equivariant orderability and realization of probability preserving actions  as IROs}
In \cref{sec:ord_ergodic_universality} we showed for any point in  $\Ord(\Gamma)$ the stabilzer subgroup of $\Gamma$ is always  orderable. We also showed that a seemingly slightly stronger property is necessary and sufficient for a (non-atomic) ergodic probability preserving action to be realizable as an IRO.

Call an invariant random subgroup $\mu \in \IRS(\Gamma)$ \emph{orderable} if  it is supported on orderable subgroups, and  \emph{equivariantly orderable} if there exists a measurable equivariant function $\pi:\Sub(\Gamma) \to \pOrd(\Gamma)$ such that for almost every $\Delta \in \Sub(\Gamma)$ (with respect to $\mu$), the restriction of $\pi(\Delta) \in \pOrd(\Gamma)$ to $\Delta$ is an invariant total order.

	
A positive answer to the following question would yield a simplified  characterization of  probability preserving actions realizable as IRO's.
\begin{quest}\label{quest:equvaiant_IRO}
Is every orderable invariant random subgroup equivariantly orderable?
\end{quest}

The argument given in the proof of \cref{orederable_realization} shows that  any IRO of an orderable group $\G$ is equivariantly orderable. Thus, a positive solution to the question below would immediately imply a positive solution to \cref{quest:equvaiant_IRO}
\begin{quest}\label{quest:IRO_orderable}
Suppose $\Gamma$ admits an ergodic IRS which is almost surely orderable and spanning in the sense that that $\Gamma$ is the smallest normal subgroup which contains all the subgroups in the support of the IRS. Is $\G$ necessarily orderable?
\end{quest}
Both ergodicity and the spanning assumptions are necessary here, as the following examples show:
\begin{example}
Let $\Gamma$ be a finitely generated group which is not left orderable, such as $\SL_3(\Z)$. By finite generation we can pick an epimorphism $\phi:F_n \rightarrow \Gamma$, where $F_n$ is the free group on $n$-generators. Let $\Delta = \{(x,x) \ | \ x \in F_n\} \lhd F_n \times F_n$ be the diagonal copy of $F_n$ in the product, and $N = \{(x,x) \ | \ x \in \ker(\phi)\}$. Now set $G = (F_n \times F_n)/N$, and $\tilde G = G \rtimes C_2$ where the cyclic group $C_2$ acts via the obvious involution $(x,y)N \mapsto (y,x) N$. Let $F_1 = \{(x,e)N \ | \ x \in F_n\}, F_2 = \{(e,x)N \ | \ x \in F_n\}$ be the (injective) images of the two free factors in $G$ and finally consider the IRS 
$$\mu = \frac{\delta_{F_1} + \delta_{F_2}}{2}.$$
This is an IRS in both groups $G < \tilde{G}$. Indeed $F_1,F_2 \lhd G$ so that the dirac measures $\delta_{F_i}$ are $G$-invariant. In $\tilde{G}$ these two are flipped by the involution so $\mu$ is still invariant. $\mu$ is supported on free groups, which are definitely left orderable, but the groups $G,\tilde{G}$ themselves fail to be left orderable because they contain an isomorphic copy of $\Gamma \cong \frac{\Delta}{N}$. These two examples just fall short from giving a counter example to Question \ref{quest:equvaiant_IRO}. As an IRS on $G$, $\mu$ is spanning but fails to be ergodic; on $\tilde{G}$, $\mu$ is ergodic, but spans only $G$. 
\end{example}

\subsection{Strong non-orderablity}
Say that a countable group $\G$ is \emph{strongly non-orderable} if every $\mu \in \IRO(\G)$ induces an  essentially free probability preserving $\G$-action. In particular, a strongly non-orderable group $\G$  does not admit  left-orderable subgroups of finite index.


The Stuck Zimmer theroem \cite{StuckZimmer94} combined with the recent proof of Hurtado and Deroin that all higher rank irreducible lattices are non-orderable \cite{deroinHurtado2020} gives rise to the following.
\begin{cor}
Let $\Gamma < G$ be an irreducible lattice in a higher rank semisimple Lie group $G$ with property (T). Then $\Gamma$ is strongly non-orderable. 
\end{cor}
\begin{proof}
Assume that $\prec \in \IRO(\Gamma)$ be an ergodic IRO which is not essentially free. The Stuck-Zimmer theorem implies that this IRO is supported on a finite orbit, and hence every order in $\supp(\mu)$ is fixed by a finite index subgroup of $\Gamma$, contradicting the main result of \cite{deroinHurtado2020}.   
\end{proof}

Finite groups are obviously strongly non-orderable, as is any normal subgroup of a strongly non-orderable group.

To the best of our knowledge, It is currently not known if a left orderable group can ever satisfy Kazhdan’s property $(T)$. 
In view of the above, we formulate the following question:
\begin{quest}\label{quest:property_T_strongly_non_orederable}
Is any group with Kazhdan’s property $(T)$ strongly non-orderable?
\end{quest}



\bibliographystyle{amsplain}
\bibliography{IRO}
\end{document}